\documentclass[11pt, letterpaper, reqno]{amsart}

\usepackage{amsmath}
\usepackage{amssymb}
\usepackage{amsfonts}
\usepackage{ dsfont }
\usepackage{cases}
\usepackage{cite}
\usepackage[hyphens]{url}
\usepackage{hyperref}
\usepackage{amsthm}
\usepackage{upgreek}

\DeclareMathOperator*{\esssup}{ess\,sup}

\newtheorem{theorem}{Theorem}
\newtheorem{corollary}{Corollary}
\newtheorem{lemma}{Lemma}
\newtheorem{proposition}{Proposition}
\newtheorem{remark}{Remark}

\newcommand{\leref}{Lemma~\ref}

\newcommand{\coref}{Corollary~\ref}

\newcommand{\prref}{Proposition~\ref}
\newcommand{\thref}{Theorem~\ref}

\newcommand{\eps}{\varepsilon}
\newcommand{\T}{\mathcal{T}}
\newcommand{\E}{\mathbb{E}}

\title[]{On an Optimal Stopping Problem of An Insider }
\author[]{Erhan Bayraktar} \thanks{This research was supported in part by the National Science Foundation under grants DMS 0955463 and DMS 1118673.}  
\address{Department of Mathematics, University of Michigan}
\email{erhan@umich.edu}
\author[]{Zhou Zhou}
\address{Department of Mathematics, University of Michigan}
\email{zhouzhou@umich.edu}
\date{\today}
\begin{document}
\maketitle

\begin{abstract}
We consider the optimal stopping problem $v^{(\eps)}:=\sup_{\tau\in\mathcal{T}_{0,T}}\mathbb{E}B_{(\tau-\eps)^+}$ posed by Shiryaev at the International Conference on Advanced Stochastic Optimization Problems organized by the Steklov Institute of Mathematics in September 2012. Here $T>0$ is a fixed time horizon, $(B_t)_{0\leq t\leq T}$ is the Brownian motion, $\eps\in[0,T]$ is a constant, and $\mathcal{T}_{\eps,T}$ is the set of stopping times taking values in $[\eps,T]$. The solution of this problem is characterized by a path dependent reflected backward stochastic differential equations, from which the continuity of $\eps \to v^{(\eps)}$ follows. For large enough $\eps$, we obtain an explicit expression for $v^{(\eps)}$ and for small $\eps$ we have lower and upper bounds.
The main result of the paper is the asymptotics of $v^{(\eps)}$ as $\eps\searrow 0$. As a byproduct, we also obtain L\'{e}vy's modulus of continuity result in the $L^1$ sense.
\end{abstract}

\section{introduction}
In this paper we consider Shiryaev's optimal stopping problem:
\begin{equation}\label{eq:value}
v^{(\eps)}=\sup_{\tau\in\mathcal{T}_{0,T}}\mathbb{E}B_{(\tau-\eps)^+},
\end{equation}
where $T>0$ is a fixed time horizon, $(B_t)_{0\leq t\leq T}$ is the Brownian motion, $\eps\in[0,T]$ is a constant, and $\mathcal{T}_{\eps,T}$ is the set of stopping times taking values in $[\eps,T]$. This can be thought of a problem of an insider in which she  is allowed to peek $\eps$ into the future for the payoff before making her stopping decision. 

We show that $v^{(\eps)}$ is the solution of a corresponding path dependent reflected backward stochastic differential equation (RBSDEs). This is essentially an existence result, and it shows that an optimal stopping time exists. But the main advantage of using an RBSDE representation is that we can easily get the continuity of $v^{(\eps)}$ with respect to $\eps$ from the stability of the RBSDEs. However, we want to compute the function as explicitly as possible, and the RBSDE representation of the problem does not  help. This is because the problem is path dependent (one of the state variables would have be an entire path of length $\eps$), and there is no numerical result available so far that can cover our case.

In fact, we will observe that $v^{(\eps)}=\sqrt{2 \left(T-\eps\right) \over \pi }$ if $\eps\in[T/2,T]$, while as far as we know there is no explicit solution for $v^{(\eps)}$ if $\eps\in(0,T/2)$.  But for smaller $\eps$, there are only lower and upper bounds available. As the main result of this paper, we provide the asymptotic behavior of $v^{(\eps)}$ as $\eps\searrow 0$ (see \thref{t3}). As a byproduct, we also get L\'{e}vy's modulus of continuity theorem in the $L^1$ sense as opposed to the almost-surely sense (compare \coref{c2} and, e.g.,\cite[Theorem 9.25, page 114]{KS}).

\section{First observations}
Let $T>0$ and let $\{B_t, t \in [0,T]\}$ be a Brownian motion defined on a probability space $(\Omega, \mathcal{F}, \mathbb{P})$ and let $\mathbb{F}=\{\mathcal{F}_t, t \in [0,T]\}$ be the natural filtration augmented by the $ \mathbb{P}$-null sets of $\mathcal{F}$.
We aim at the problem \eqref{eq:value}. But for the sake of generality, let us first look at the more general optimal stopping problem of an insider:
\begin{equation}\label{eq:general}w=\sup_{\tau\in\mathcal{T}_{\eps,T}}\mathbb{E}\left[\sum_{i=1}^n \phi_{(\tau-\eps^i)+}^i\right],\end{equation}
where $(\phi_t^i)_{0\leq t\leq T}$ is continuous and progressively measurable, $\eps^i\in[0,T]$, $i=1,\dotso,n$, are given constants, and $\mathcal{T}_{\eps,T}$ is the set of stopping times that lie between a constant $\eps\in[0,T]$ and $T$. Observe that $\tau-\eps^i$ is not a stopping time with respect to $\mathbb{F}$ for $\eps^i>0$. The solution to  \eqref{eq:general} is described by the following result:
\begin{proposition}\label{thm:Main}
Assume $\mathbb{E}\big[\sup_{0\leq t\leq T}(\xi_t^+)^2\big]<\infty$, where $\xi_t=\sum_{i=1}^n \phi_{(t-\eps^i)+}^i,\ 0\leq t\leq T$. Then the value defined in \eqref{eq:general} can be calculated using a reflected backward stochastic differential equation (RBSDE). More precisely, $w=\mathbb{E}Y_\eps$, for any $\eps \in [0,T]$,
where $(Y_t)_{0\leq t\leq T}$ satisfies the RBSDE
\begin{equation}\label{eq:RBSDE}
\begin{split}
 &\xi_t \leq Y_t=\xi_T-\int_t^TZ_sdW_s+(K_T-K_t), \; 0\leq t\leq T,\\
  & \int_0^T(Y_t- \xi_t)dK_t=0, 
\end{split}
\end{equation}
Moreover, there exists an optimal stopping time $\hat\tau^{}$ described by
$$\hat\tau=\inf\{t \in [\eps,T]: Y_t=\xi_t\}.$$ 
\end{proposition}
\begin{remark}
One should note that the optimal stopping problem we are considering is path dependent (i.e. not of Markovian type) and therefore one would not be able to write down a classical free boundary problem corresponding to \eqref{eq:value}.
\end{remark}
We prefer to use an RBSDE representation of the value function instead of directly using the representation directly from the classical optimal stopping theory because we want to use the stability result, which we will state in Corollary~\ref{co1}, associated with the former.
\begin{proof}[Proof of \prref{thm:Main}]
For any $\tau\in\mathcal{T}_{\eps,T}$, 
$$\mathbb{E}\xi_\tau=\mathbb{E}[\mathbb{E}[\xi_\tau|\mathcal{F}_{\eps}]]\leq\mathbb{E}\Big[\esssup_{\sigma\in\mathcal{T}_{\eps,T}}\mathbb{E}[\xi_\sigma|\mathcal{F}_{\eps}]\Big].$$
Therefore,
\begin{equation}\label{4}w=\sup_{\tau\in\mathcal{T}_{\eps,T}}\mathbb{E}\xi_\tau\leq \mathbb{E}\Big[\esssup_{\tau\in\mathcal{T}_{\eps,T}}\mathbb{E}[\xi_\tau|\mathcal{F}_{\eps}]\Big].\end{equation}
By Theorem 5.2 in \cite{El1} there exists a unique solution $(Y,Z,K)$ to the RBSDE in \eqref{eq:RBSDE}. Then by Proposition 2.3 (and its proof) in \cite{El1} we have
$$\sup_{\tau\in\mathcal{T}_{\eps,T}}\mathbb{E}\xi_\tau\geq\mathbb{E}\xi_{\hat\tau}=\mathbb{E}Y_{\hat\tau}=\mathbb{E}Y_{\eps}=\mathbb{E}\Big[\esssup_{\tau\in\mathcal{T}_{\eps,T}}\mathbb{E}[\xi_\tau|\mathcal{F}_{\eps}]\Big].$$
Along with \eqref{4} the last inequality completes the proof. 
\end{proof}

Now let us get back to Shiryaev's problem \eqref{eq:value}. As a corollary of \prref{thm:Main}, we have the following result for $v^{(\eps)},\ \eps \in [0,T]$. 
\begin{corollary}\label{co1} 
The value defined in \eqref{eq:value} can be calculated using an RBSDE. More precisely, $v^{\eps}=Y_0$ almost surely,
where $(Y_t)_{0\leq t\leq T}$ satisfies the RBSDE \eqref{eq:RBSDE} with $\xi$ defined as $\xi_t=B_{(t-\eps)^+},\ 0\leq t\leq T$. Moreover, there exists an optimal stopping time $\tilde\tau$ described by
\begin{equation}\label{stop}\tilde\tau=\inf\{t\geq 0: Y_t=B_{(t-\eps)^+}\}\geq\eps 1_{\{\eps<T\}},\quad\text{a.s.}.\end{equation}
Furthermore, the function $\eps \to v^{(\eps)}$, $\eps \in [0,T]$, is a continuous function.
\end{corollary}
\begin{proof}
By \prref{thm:Main} $v^{(\eps)}=Y_0$ a.s., and $\tilde\tau$ defined in \eqref{stop} is optimal. Besides, the continuity of $\eps \to v^{(\eps)}$, $\eps \in [0,T]$ is a direct consequence of the stability of RBSDEs indicated by Proposition 3.6 in  \cite{El1}. Observe that for $\eps\in(0,T)$ and $t\in[0,\eps]$, $Y_t\geq\mathbb{E}[Y_\eps|\mathcal{F}_t]>0=B_{(t-\eps)^+}$ a.s.. Hence we have that $\tilde\tau\geq\eps 1_{\{\eps<T\}}$ a.s..
\end{proof}
\begin{remark}
In the above result, since for any $\delta\in[0,\eps]$
$$v^{(\eps)}=\sup_{\tau\in\mathcal{T}_{0,T}}\mathbb{E}B_{(\tau-\eps)^+}=\sup_{\tau\in\mathcal{T}_{\delta,T}}\mathbb{E}B_{(\tau-\eps)^+},$$
we can conclude from \prref{thm:Main} that $v^{(\eps)}=\mathbb{E}Y_\delta$, which implies that $(Y_t)_{t \in [0,\eps]}$ is a martingale. 
\end{remark}

Next, we will make some observations about the magnitude of the function $\eps \to v^{(\eps)}$:
\begin{remark}
Observe that for $\eps\in(0,T)$, insider's value defined in \eqref{eq:value} is strictly greater than 0 (and hence does strictly better than a stopper which does not posses the insider information):
$$v^{(\eps)}\geq\mathbb{E}\Big[\max_{0\leq t\leq \eps\wedge(T-\eps)} B_t\Big]=\sqrt{{2 \over \pi}\left(\eps\wedge(T-\eps)\right)}>v^{(0)}=0,$$
which shows that there is an incentive for waiting. We also have an upper bound
\[
v^{(\eps)} \leq \mathbb{E}\Big[\max_{0 \leq t \leq T} B_t\Big] = \sqrt{2 T \over \pi}.
\]
In fact when $\eps \in [T/2,T]$, $v^{(\eps)}$ can be explicitly determined
as
\[
v^{(\eps)} = \mathbb{E}\left[\max_{0 \leq t \leq T-\eps}B_{t} \right]=\sqrt{2 \left(T-\eps\right) \over \pi }, \quad \eps \in [T/2,T].
\] 
and we have a strict lower bound for $\eps \in [0, T/2)$
\[
v^{(\eps)} >\mathbb{E}\left[\max_{0 \leq t \leq \eps}B_{t} \right]=\sqrt{2 \eps \over \pi }, \quad \eps \in [0,T/2).
\]
\end{remark}

\section{Asymptotic behavior of $v^{(\eps)}$ as $\eps\searrow 0$}
The following theorem states that the order of $v^{(\eps)}$ defined in \eqref{eq:value} is $\sqrt{2\eps\ln(1/\eps)}$ as $\eps\searrow 0$, which is the same as Levy's modulus for Brownian motion.  Notice that
$$v^{(\eps)}=\sup_{\tau\in\mathcal{T}_{\eps,T}}\mathbb{E}[B_{\tau-\eps}-B_\tau].$$
\begin{theorem}\label{t3}
\begin{equation}\label{e3}
\lim_{\eps\searrow 0}\frac{v^{(\eps)}}{\sqrt{2\eps\ln(1/\eps)}}=1.
\end{equation}
\end{theorem}
In order to prepare the proof of the theorem, we will need two lemmas.
\begin{lemma}
$$\liminf_{\eps\searrow 0}\frac{v^{(\eps)}}{\sqrt{2\eps\ln(1/\eps)}}\geq 1.$$
\end{lemma}
\begin{proof}
Let $d\in(0,1)$ be a constant, and define $\tau^*\in\T_{\eps,T}$
$$\tau^*:=\inf\{n\eps:\ B_{(n-1)\eps}-B_{n\eps}\geq d\sqrt{2\eps\ln(1/\eps)},\ n=1,\dotso,[T/\eps]-1\}\wedge T.$$
Then
\begin{eqnarray}
\notag\sup_{\tau\in\mathcal{T}_{\eps,T}}\mathbb{E}[B_{\tau-\eps}-B_\tau]&\geq&\mathbb{E}[B_{\tau^*-\eps}-B_{\tau^*}]\\
\notag&=&\mathbb{E}\left[\left(B_{\tau^*-\eps}-B_{\tau^*}\right)1_{\{\tau^*\leq\eps[T/\eps]-\eps\}}\right]+\mathbb{E}\left[\left(B_{\tau^*-\eps}-B_{\tau^*}\right)1_{\{\tau^*>\eps[T/\eps]-\eps\}}\right]\\
\notag&\geq&d\sqrt{2\eps\ln(1/\eps)}\,P(\tau^*\leq\eps[T/\eps]-\eps)+\mathbb{E}\left[\left(B_{T-\eps}-B_T\right)1_{\{\tau^*>\eps[T/\eps]-\eps\}}\right]\\
\notag&=&d\sqrt{2\eps\ln(1/\eps)}\,P(\tau^*\leq\eps[T/\eps]-\eps).
\end{eqnarray}
We have that
\begin{eqnarray}
\notag P(\tau^*\leq\eps[T/\eps]-\eps)&=&1-P\left(B_{(n-1)\eps}-B_{n\eps}<d\sqrt{2\eps\ln(1/\eps)},\ n=1,\dotso,[T/\eps]-1\right)\\
\notag&=&1-\left[P\left(B_\eps-B_0<d\sqrt{2\eps\ln(1/\eps)}\right)\right]^{[T/\eps]-1}\\
\notag&=&1-\left[1-\int_{d\sqrt{2\eps\ln(1/\eps)}}^\infty\frac{1}{\sqrt{2\pi\eps}}e^{-\frac{x^2}{2\eps}}dx\right]^{[T/\eps]-1}\\
\notag&=&1-(1-\alpha)^{\frac{1}{\alpha}([T/\eps]-1)\alpha},\\
\end{eqnarray}
where
$$\alpha:=\int_{d\sqrt{2\eps\ln(1/\eps)}}^\infty\frac{1}{\sqrt{2\pi\eps}}e^{-\frac{x^2}{2\eps}}dx=\frac{1}{2d\sqrt{\pi\ln(1/\eps)}}\eps^{d^2}(1+o(1))\rightarrow 0,$$
by, e.g., \cite[(9.20) on page 112]{KS}. Since $d\in(0,1)$, $([T/\eps]-1)\alpha\rightarrow\infty$, and thus
$$P(\tau^*\leq\eps[T/\eps]-\eps)\rightarrow 1,\quad\eps\searrow 0.$$
Therefore,
$$\liminf_{\eps\searrow 0}\frac{v^{(\eps)}}{\sqrt{2\eps\ln(1/\eps)}}\geq\liminf_{\eps\searrow 0}\left[d\,P(\tau^*\leq\eps[T/\eps]-\eps)\right]=d.$$
Then \eqref{e3} follows by letting $d\nearrow 1$.
\end{proof}
\begin{lemma}\label{l5}
The family
$$\left\{\frac{\sup_{\eps\leq t\leq T}|B_{t-\eps}-B_t|}{\sqrt{2\eps\ln(1/\eps)}}:\ \eps\in\left(0,\frac{T\wedge 1}{2}\right]\right\}$$
 is uniformly integrable.
\end{lemma}
\begin{proof}
Since
\begin{eqnarray}
\notag\frac{\sup_{\eps\leq t\leq T}|B_{t-\eps}-B_t|}{\sqrt{2\eps\ln(1/\eps)}}&\leq&\frac{2\max_{1\leq n\leq [T/\eps]+1}\sup_{(n-1)\eps\leq t,t'\leq n\eps}|B_t-B_{t'}|}{\sqrt{2\eps\ln(1/\eps)}}\\
\notag&\leq&\frac{4\max_{1\leq n\leq [T/\eps]+1}\sup_{(n-1)\eps\leq t\leq n\eps}|B_t-B_{(n-1)\eps}|}{\sqrt{2\eps\ln(1/\eps)}},
\end{eqnarray}
it suffices to show that the family
$$\left\{M_\eps:=\frac{\max_{1\leq n\leq [T/\eps]+1}\sup_{(n-1)\eps\leq t\leq n\eps}|B_t-B_{(n-1)\eps}|}{\sqrt{\eps\ln(1/\eps)}}:\ \eps\in\left(0,\frac{T\wedge 1}{2}\right]\right\}$$
is uniformly integrable. For $a\geq 0$,
$$P(M_\eps\leq a)=\left[P\left(\sup_{0\leq t\leq\eps}|B_t|\leq a\,\sqrt{\eps\ln(1/\eps)}\right)\right]^{[T/\eps]+1}.$$
Hence the density of $M_\eps$, $f_\eps$, satisfies that for $a\geq 0$,

\begin{eqnarray}
\notag f_\eps(a)&\leq&([T/\eps]+1)\left[P\left(\sup_{0\leq t\leq\eps}|B_t|\leq a\,\sqrt{\eps\ln(1/\eps)}\right)\right]^{[T/\eps]}\sqrt{\frac{8}{\pi}}\sqrt{\ln(1/\eps)}e^{-\frac{\ln(1/\eps)}{2}a^2}\\
\notag &\leq&\frac{4T\sqrt{\ln(1/\eps)}}{\eps}e^{-\frac{\ln(1/\eps)}{2}a^2},
\end{eqnarray}
where for the first inequality we use, e.g., \cite[(8.3) on page 96]{KS}, and the fact that the density of $\sup_{0\leq t\leq\eps}|B_t|$ is no greater than twice the density of $\sup_{0\leq t\leq\eps}B_t$. Then we have that for $N>0$,
$$\E\left[M_\eps1_{\{M_\eps>N\}}\right]=\int_N^\infty xf_\eps(x)dx\leq \frac{4T\sqrt{\ln(1/\eps)}}{\eps}\int_N^\infty xe^{-\frac{\ln(1/\eps)}{2}x^2}dx=\frac{4T\eps^{\frac{N^2}{2}-1}}{\sqrt{\ln(1/\eps)}}\leq\frac{T}{2^{\frac{N^2}{2}-3}\sqrt{\ln 2}},$$
i.e.,
$$\lim_{N\rightarrow\infty}\sup_{\eps\in(0,\frac{T\wedge 1}{2}]}\E\left[M_\eps1_{\{M_\eps>N\}}\right]=0.$$
\end{proof}
Now let us turn to the proof of \thref{t3}.
\begin{proof}[Proof of \thref{t3}]
\begin{eqnarray}
\notag\limsup_{\eps\searrow 0}\frac{\sup_{\tau\in\mathcal{T}_{\eps,T}}\mathbb{E}[B_{\tau-\eps}-B_\tau]}{\sqrt{2\eps\ln(1/\eps)}}&\leq&\limsup_{\eps\searrow 0}\E\left[\frac{\sup_{\eps\leq t\leq T}|B_{t-\eps}-B_t|}{\sqrt{2\eps\ln(1/\eps)}}\right]\\
\notag&\leq&\E\left[\limsup_{\eps\searrow 0}\frac{\sup_{\eps\leq t\leq T}|B_{t-\eps}-B_t|}{\sqrt{2\eps\ln(1/\eps)}}\right]\\
\notag&\leq&1,
\end{eqnarray}
where we apply \leref{l5} for the second inequality, and use Levy's modulus for Brownian motion (see, e.g., \cite[Theorem 9.25, page 114]{KS}) for the third inequality. Together with \eqref{e3}, the conclusion follows. 
\end{proof}
Using the above proof, we can actually show the following result, which is L\'{e}vy's modulus continuity result in the $L^1$ sense, as opposed to the almost-surely sense (see, e.g., \cite[Theorem 9.25, page 114]{KS}).
\begin{corollary}\label{c2}
$$\lim_{\eps\searrow 0}\frac{\sup_{\tau\in\mathcal{T}_{\eps,T}}\mathbb{E}[B_{\tau-\eps}-B_\tau]}{\sqrt{2\eps\ln(1/\eps)}}=\lim_{\eps\searrow 0}\E\left[\frac{\sup_{\eps\leq t\leq T}(B_{t-\eps}-B_t)}{\sqrt{2\eps\ln(1/\eps)}}\right]=\lim_{\eps\searrow 0}\E\left[\frac{\sup_{\eps\leq t\leq T}|B_{t-\eps}-B_t|}{\sqrt{2\eps\ln(1/\eps)}}\right]=1.$$
\end{corollary}

\bibliographystyle{siam}
\bibliography{ref}

\end{document}